\documentclass[12pt,reqno]{amsart}
\usepackage{indentfirst, amssymb, amsmath, amsthm, mathrsfs, setspace, indentfirst, enumerate,  mathrsfs, amsmath, amsthm}
\usepackage[bookmarksnumbered, colorlinks, plainpages]{hyperref}
\usepackage{mathrsfs}
\usepackage{cite}
\usepackage{color,xcolor}

\newtheorem*{theoA}{Theorem A}
\newtheorem*{theoB}{Theorem B}
\newtheorem*{theoC}{Theorem C}
\newtheorem*{theoD}{Theorem D}
\newtheorem*{theoE}{Theorem E}

\newtheorem*{lemA}{Lemma A}

\newtheorem*{cor A}{Corollary A}
\newtheorem*{cor B}{Corollary B}

\newtheorem{theo}{Theorem}[section]
\newtheorem{lem}{Lemma}[section]

\newtheorem{prob}{Problem}[section]

\newcommand{\ol}{\overline}
\newcommand{\be}{\begin{equation}}
\newcommand{\ee}{\end{equation}}
\newcommand{\beas}{\begin{eqnarray*}}
\newcommand{\eeas}{\end{eqnarray*}}
\newcommand{\bea}{\begin{eqnarray}}
\newcommand{\eea}{\end{eqnarray}}

\numberwithin{equation}{section}
\begin{document}
\title[M\MakeLowercase{ultidimensional analogues of the improved} B\MakeLowercase{ohr's} \MakeLowercase{inequality}]{\LARGE M\MakeLowercase{ultidimensional analogues of the improved} B\MakeLowercase{ohr's} \MakeLowercase{inequality}}
\date{}
\author[M. B. A\MakeLowercase{hamed}, S. M\MakeLowercase{ajumder}, N. S\MakeLowercase{arkar} \MakeLowercase{and} M.-S. L\MakeLowercase{iu} ]{M\MakeLowercase{olla} B\MakeLowercase{asir} A\MakeLowercase{hamed}$^*$, S\MakeLowercase{ujoy} M\MakeLowercase{ajumder}, N\MakeLowercase{abadwip} S\MakeLowercase{arkar} \MakeLowercase{and} M\MakeLowercase{ing}-S\MakeLowercase{heng} L\MakeLowercase{iu}}

\address{Molla Basir Ahamed,
	Department of Mathematics,
	Jadavpur University,
	Kolkata-700032, West Bengal, India.}
\email{mbahamed.math@jadavpuruniversity.in}

\address{Sujoy Majumder, Department of Mathematics, Raiganj University, Raiganj, West Bengal-733134, India.}
\email{sm05math@gmail.com}

\address{Nabadwip Sarkar, Department of Mathematics, Raiganj University, Raiganj, West Bengal-733134, India.}
\email{naba.iitbmath@gmail.com}

\address{Ming-Sheng Liu, School of Mathematical Sciences, South China Normal University, Guangzhou, Guangdong 510631, China.}
\email{liumsh65@163.com}

\renewcommand{\thefootnote}{}
\footnote{*\emph{Corresponding Author}: Molla Basir Ahamed.}

\renewcommand{\thefootnote}{\arabic{footnote}}
\setcounter{footnote}{0}

\begin{abstract}
The main aim of this article is to establish a sharp improvement of the classical Bohr inequality for bounded holomorphic mappings in the polydisk $\mathbb{P}\Delta(0;1_n)$. We also prove two other sharp versions of the Bohr inequality in the setting of several complex variables by replacing the constant term with the absolute value of the function and the square of the absolute value of the function, respectively. All the results are shown to be sharp.
\end{abstract}
\thanks{Typeset by \AmS -\LaTeX}
\maketitle

\section{\bf Preliminaries and some basic questions}

The classical theorem of Harald Bohr \cite{Bohr-PLMS-2014}, originally examined a century ago, continues to generate intensive research on what is now known as Bohr's phenomenon. Renewed interest surged in the 1990s due to successful extensions to holomorphic functions of several complex variables and to more abstract functional analytic settings.For instance, in $1997$, Boas and Khavinson \cite{Boas-Khavinson-PAMS-1997} introduced and determined the $n$-dimensional Bohr radius for the family of holomorphic functions bounded by unity on the polydisk (see Section \ref{Sub-Sec-1.3} for a detailed discussion). This seminal work stimulated significant research interest in Bohr-type questions across diverse mathematical domains.Subsequent investigations have yielded further results on Bohr's phenomenon for multidimensional power series. Notable contributions in this area include those by Aizenberg \cite{Aizen-PAMS-2000,Aizenberg,Aizenberg-SM-2007}, Aizenberg \textit{et al.} \cite{Aizenberg-PAMC-1999,Aizenberg-SM-2005,Aizenberg-Aytuna-Djakov-JMAA-2001}, Defant and Frerick \cite{Defant-Frerick-IJM-2006}, and Djakov and Ramanujan \cite{Djakov-Ramanujan-JA-2000}. A comprehensive overview of various aspects and generalizations of Bohr's inequality can be found in \cite{Lata-Singh-PAMS-2022,Liu-Pon-PAMS-2021,Ali-Abu-Muhanna-Ponnusamy,Alkhaleefah-Kayumov-Ponnusamy-PAMS-2019,Defant-Frerick-AM-2011,Hamada-IJM-2012,Paulsen-PLMS-2002,Paulsen-Singh-PAMS-2004,Paulsen-BLMS-2006}, the monographs by Kresin and Maz'ya \cite{Kresin-1903}, and the references cited therein. In particular, \cite[Section 6.4]{Kresin-1903} on Bohr-type theorems highlights rich opportunities to extend several existing inequalities to holomorphic functions of several complex variables and, significantly, to solutions of PDEs.
\subsection{\bf Classical Bohr inequality and its recent implications}
Let $ \mathcal{B} $ denote the class of analytic functions in the unit disk $ \mathbb{D}:=\{z\in\mathbb{C} : |z|<1\} $ of the form $ f(z)=\sum_{k=0}^{\infty}a_k z^k $ such that $ |f(z)|<1 $ in $ \mathbb{D} $. In the study of Dirichlet series, in $ 1914$, Harald Bohr \cite{Bohr-PLMS-2014} discovered the following interesting phenomenon that if $ f\in\mathcal{B} $, then its associated majorant series
\begin{align}
	\label{b1}	M_f(r):=\sum_{k=0}^{\infty}|a_k|r^k\leq 1\;\;\mbox{for}\;\; |z|=r\leq \frac{1}{3}.
\end{align}
The constant $ 1/3 $, known as the Bohr radius for the class $ \mathcal{B} $, is the best possible. Furthermore, for $ \psi_a(z)={(a-z)}/{(1-az)},\; a\in [0,1), $ it follows easily that $ M_{\psi_a}(r)>1 $ if, and only if, $r>1/(1+2a)$ and hence the radius $ 1/3 $ is optimal as $ a\rightarrow 1 $. For other proofs of this theorem, we refer to the articles \cite{Bohr-PLMS-2014,Sidon-MZ-1927,Tomic-MS-1962}. It is worth pointing out that there is no extremal function in $\mathcal{B}$ such that the Bohr radius is precisely ${1/3}$ (see \cite{Garcia-Mashreghi-2018}). \vspace{1.2mm}

This result has generated substantial interest in Bohr's inequality across various settings; see, for instance, \cite{Alkhaleefah-Kayumov-Ponnusamy-PAMS-2019,Boas-Khavinson-PAMS-1997}, \cite{Paulsen-Singh-PAMS-2004}, \cite{Popescu-TAMS-2007} and the recent survey by Ali \textit{et al.} \cite{Ali-Abu-Muhanna-Ponnusamy}, \cite{Ismagilov-Kayumova-Kayumov-Ponnusamy}, \cite[Chapter 8]{Garcia-Mashreghi-Ross}, along with the references therein. Several of these works employ methods from complex analysis, functional analysis, number theory, and probability, and they further develop new theory and applications of Bohr's ideas on Dirichlet series. For example, various multidimensional generalizations of this result have been obtained in \cite{Aizenberg-PAMC-1999,Aizenberg-Aytuna-Djakov-JMAA-2001,Aizenberg,Boas-Khavinson-PAMS-1997,Djakov-Ramanujan-JA-2000,Jia-Liu-Ponnusamy-AMP-2025}.\vspace{1.2mm}

Our primary interest in this paper is to establish several multidimensional analogues of the improved Bohr's inequality for the analytic case and to prove that these results are sharp.
\subsection{The improved versions of the classical Bohr's inequality.} Throughout this paper $S_r(f)$ denotes the area of the image of the subdisk $|z|<r$ under the mapping $f$. If $f(z)=\sum_{k=0}^{\infty}a_k z^k$, then from the definition of $S_r$, we see that
\[\frac{S_r}{\pi}=\frac{1}{\pi}\int\int_{|z|<r} |f'(z)|^2\;dx\;dy=\sum_{k=1}k|a_k|^2r^{2k}.\]
Recently, Kayumov and Ponnusamy \cite{Kayumov-Ponnusamy-CRMASP-2018} improved the classical Bohr inequality and established the result.
\begin{theoA}Suppose that $f(z)=\sum_{k=0}^{\infty} a_k z^k$  is analytic in $\mathbb{D}$ and $|f(z)| \le 1$ in $\mathbb{D}$. Then
\begin{align}
	\label{b2}|a_0| + \sum_{k=1}^{\infty} |a_k| r^k + \frac{16}{9} \left( \frac{S_r}{\pi} \right) \le 1 \quad \text{for } |z|=r \le \frac{1}{3}
\end{align}
and the numbers $1/3$ and $16/9$ cannot be improved.
\end{theoA}
It is worth recalling that if the first term $|a_0|$ in (\ref{b2}) is replaced by $|a_0|^2$, the constants ${1/3}$ and ${16/9}$ are replaced by ${1/2}$ and ${9/8}$, respectively (see \cite{Kayumov-Ponnusamy-CRMASP-2018}). More recently, the authors of \cite{Ismagilov-Kayumova-Kayumov-Ponnusamy} improved Theorem A by replacing the quantity $S_r/\pi$ in inequality (\ref{b2}) with $S_r/(\pi-S_r)$. Further generalizations of Bohr's results were also obtained (see \cite[Corollary 3]{Ismagilov-Kayumova-Kayumov-Ponnusamy}). Subsequently, Liu \textit{et al.} \cite{Liu-Shang-Xu-JIA-2018} employed the ideas developed in \cite{Ismagilov-Kayumova-Kayumov-Ponnusamy} to advance the study of Bohr's inequality.
\begin{theoB} Suppose that $f(z)=\sum_{k=0}^{\infty} a_k z^k$  is analytic in $\mathbb{D}$ and $|f(z)| \le 1$ in $\mathbb{D}$. Then
\begin{enumerate}
	\item[\emph{(a)}] $|f(z)| + \sum_{k=1}^{\infty} |a_k| r^k \le 1 \quad \text{for } |z| = r \le \sqrt{5} - 2 \approx 0.236068.$ The constant $\sqrt{5} - 2$ is the best possible. Moreover,
	\item[\emph{(b)}] $|f(z)|^{2} + \sum_{k=1}^{\infty} |a_k| r^k \le 1 \quad \text{for } |z| = r \le \frac{1}{3}.$ The constant $\frac{1}{3}$ is the best possible.
\end{enumerate}
\end{theoB}
Ismagilov \textit{et al.} \cite{Ismagilov-Kayumova-Kayumov-Ponnusamy-JMAA-2020} derived a sharp version of Theorems A and B in ${2020}$. These results reveal the presence of an extremal function confirming that the value ${1/3}$ is the maximum constant (i.e., it cannot be increased).
\begin{theoC} Suppose that $f(z)=\sum_{k=0}^{\infty} a_k z^k$  is analytic in $\mathbb{D}$ and $|f(z)| \le 1$ in $\mathbb{D}$. Then
\begin{align}
\label{b3}M(r) := \sum_{k=0}^{\infty} |a_k| r^k 
+ \frac{16}{9} \left( \frac{S_r}{\pi} \right)
+ \lambda_1 \left( \frac{S_r}{\pi} \right)^{2}
\le 1 \quad \text{for } r \le 1/3,
\end{align}
where $\lambda_1 
= \frac{4 \left( 486 - 261a - 324a^{2} + 2a^{3} + 30a^{4} + 3a^{5} \right)}
{81 (1 + a)^{3} (3 - 5a)}
= 18.6095\ldots,$
and $a \approx 0.567284$ is the unique root in $(0,1)$ of the equation $\psi_1(t) = 0$, where $\psi_1(t)= -405 + 473 t + 402 t^{2} + 38 t^{3} + 3 t^{4} + t^{5}.$
The equality is achieved for the function $\psi_a(z)$.
\end{theoC}
The next result follows from the fact that there is a concrete extremal function for which equality in (\ref{b3}) holds, meaning no strictly positive terms can be added.
\begin{theoD} Suppose that $f(z)=\sum_{k=0}^{\infty} a_k z^k$  is analytic in $\mathbb{D}$ and $|f(z)| \le 1$ in $\mathbb{D}$. Then
\begin{align*}
M(z, r) := |f(z)|^{2}
+ \sum_{k=1}^{\infty} |a_k| r^{k}
+ \frac{16}{9} \left( \frac{S_r}{\pi} \right)
+ \lambda_2 \left( \frac{S_r}{\pi} \right)^{2}
\le 1 \quad \text{for } |z| = r \le 1/3,
\end{align*}
where $\lambda_2 = 
\frac{-81 + 1044a + 54a^{2} - 116a^{3} - 5a^{4}}
{162 (a + 1)^{2} (2a - 1)}
= 16.4618\ldots,$ and $a \approx 0.537869$ is the unique root in $(0,1)$ of the equation $\psi_2(t)=0$, where $\psi_2(t)=-513 + 910t + 80t^{2} + 2t^{3} + t^{4}$.
The equality is achieved for the function $\psi_a(z)$.
\end{theoD}	
Ismagilov $\textit{et al.}$ \cite{Ismagilov-Kayumova-Kayumov-Ponnusamy-JMAA-2020} replaced $|f(z)|^2$ by $|f(z)|$ in Theorem D; however, this resulted in a smaller Bohr radius. 
\begin{theoE} Suppose that $f(z)=\sum_{k=0}^{\infty} a_k z^k$  is analytic in $\mathbb{D}$ and $|f(z)| \le 1$ in $\mathbb{D}$. Then
\begin{align*}
|f(z)| + \sum_{k=1}^{\infty} |a_k| r^k + p \left(\frac{S_r}{\pi} \right)
\le 1 \quad \text{for } |z| = r \le \sqrt{5} - 2 \approx 0.236068.
\end{align*}
The constants $r_0 = \sqrt{5} - 2$ and $p = 2(\sqrt{5} - 1)$ are sharp.
\end{theoE}	
\subsection{\bf Basic Notations in several complex variables}\label{Sub-Sec-1.3}
For $z=(z_1,\ldots,z_n)$ and $w=(w_1,\ldots,w_n)$ in $\mathbb{C}^{n}$, we denote $\langle z,w\rangle=z_1\ol w_1+\ldots+z_n \ol w_n$ and $||z||=\sqrt{\langle z,z\rangle}$. The absolute value of a complex number $z_1$ is denoted by $|z_1|$ and for $z\in\mathbb{C}^n$, we define $||z||_{\infty}=\max\limits_{1\leq i\leq n}|z_i|.$ An open polydisk (or open polycylinder) in $\mathbb{C}^n$ is a subset $\mathbb{P}\Delta(a;r)\subset \mathbb{C}^n$ of the form 
\[\mathbb{P}\Delta(a;r)=\sideset{}{_{j=1}^n}{\prod} \Delta(a_j;r_j)=\lbrace z\in\mathbb{C}^n: |z_i-a_i|<r_i,\;i=1,2,\ldots,n\rbrace,\]
the point $a=(a_1,\ldots,a_n)\in\mathbb{C}^n$ is called the centre of the polydisk and $r=(r_1,\ldots,r_n)\in\mathbb{R}^n\;(r_i>0)$ is called the polyradius. It is easy to see that
\begin{align*}
	\mathbb{P}\Delta(0;1)=\mathbb{P}\Delta(0;1_n)=\sideset{}{_{j=1}^n}{\prod} \Delta(0_n;1_n).
\end{align*}
A multi-index $\alpha=(\alpha_1,\ldots,\alpha_n)$ of dimension $n$ consists of n non-negative integers $\alpha_j,\;1\leq j\leq n$; the degree of a multi-index $\alpha$ is the sum $|\alpha|=\sum_{j=1}^n \alpha_j$ and we denote $\alpha!=\alpha_1!\ldots \alpha_n!$. For $z=(z_1,\ldots,z_n)\in\mathbb{C}^n$ and a multi-index $\alpha=(\alpha_1,\ldots,\alpha_n)$, we define 
\begin{align*}
	z^{\alpha}=\sideset{}{_{j=1}^n}{\prod} z_j^{\alpha_j}\;\;\text{and}\;\;|z|^{\alpha}=\sideset{}{_{j=1}^n}{\prod} |z_j|^{\alpha_j}.
\end{align*}
It is natural to raise the following.
\begin{prob}\label{P-1}
	Can we establish the multidimensional versions of Theorems C-E?
\end{prob}
The primary goal of this article is to provide an affirmative answer to Problem \ref{P-1}. The paper is organized as follows: In Section \ref{Sec-2}, we state the main results of this paper. In Section \ref{Sec-3}, we prove several lemmas that will play a key role in establishing our main results. Section \ref{Sec-4} then presents the detailed proofs of these main results.
\section{{\bf Main Results}}\label{Sec-2}
Let $f(z)=\sum_{|\alpha|=0} a_{\alpha}z^{\alpha}$ be a holomorphic function in the polydisk $\mathbb{P}\Delta(0;1_n)$ such that $|f(z)|\leq 1$ for all $z\in \mathbb{P}\Delta(0;1/n)$ and define $M_f(|z|)=\sum_{|\alpha|=0}^{\infty} |a_{\alpha}| |z|^{\alpha}$. We now state our first result, which is a multidimensional version of Theorem C.
\begin{theo}\label{Th-2.1} Let $f(z)=\sum_{|\alpha|=0} a_{\alpha}z^{\alpha}$ be a holomorphic function in the polydisk $\mathbb{P}\Delta(0;1_n)$ such that $|f(z)|\leq 1$ for all $z\in \mathbb{P}\Delta(0;1/n)$. Then
\begin{align}
\label{2.1} \mathcal{A}_f(r):=M_f(|z|)+\frac{16}{9}\sum_{k=1}^{\infty}k\sum_{|\alpha|=k}|a_{\alpha}|^2\;r^{2\alpha}+\lambda_1 \left(\sum_{k=1}^{\infty}k\sum_{|\alpha|=k}|a_{\alpha}|^2\;r^{2\alpha}\right)^2\leq 1,
\end{align}
for ${\bf r}\leq 1/3n$, where $z=(z_1,\ldots,z_n)\in \mathbb{P}\Delta(0;1/n)$ and $r=(r_1,r_2,\ldots,r_n)$ such that ${\bf r}=||z||_{\infty}$, where $\lambda_1$, $|a_0|=a \approx 0.567284$ is the unique root of the equation $\psi_1(t)= 0$, are defined in Theorem C.
The equality is attained by the holomorphic function $f_a$ in the polydisk $\mathbb{P}\Delta(0;1/n)$ 
\begin{align}\label{Eq-2.2}
	f_a(z)=\frac{a-(z_1+z_2+\cdots+z_n)}{1-a(z_1+z_2+\cdots+z_n)},\; \mbox{for}\; a\in [0, 1).
\end{align}
\end{theo}
Next, we state our second result, which is a multidimensional version of Theorem D.
\begin{theo}\label{Th-2.2} Let $f(z)=\sum_{|\alpha|=0} a_{\alpha}z^{\alpha}$ be a holomorphic function in the polydisk $\mathbb{P}\Delta(0;1_n)$ such that $|f(z)|\leq 1$ for all $z\in \mathbb{P}\Delta(0;1/n)$. Then
\begin{align}\label{2.2} \mathcal{B}(z,r):=|f(z)|^2+\mathcal{A}_f(r)\leq 1,\; \mbox{for}\; {\bf r}\leq \frac{1}{3n},
\end{align}
where $z=(z_1,\ldots,z_n)\in \mathbb{P}\Delta(0;1/n)$ and $r=(r_1,r_2,\ldots,r_n)$ such that ${\bf r}=||z||_{\infty}$, where $\lambda$, $a \approx 0.537869$ is the unique root in $(0, 1)$ of the equation $\psi_2(t)=0$, are defined as in Theorem D. The equality is attained by the function $f_a$ as defined in \eqref{Eq-2.2}.
\end{theo}
Our final result is the multidimensional generalization of Theorem E.
\begin{theo}\label{Th-2.3} Let $f(z)=\sum_{|\alpha|=0} a_{\alpha}z^{\alpha}$ be a holomorphic function in the polydisk $\mathbb{P}\Delta(0;1_n)$ such that $|f(z)|\leq 1$ for all $z\in \mathbb{P}\Delta(0;1/n)$. Then
\begin{align}
\label{2.3} \mathcal{C}_f(r):=|f(z)|+\mathcal{A}_f(r)+p\left(\sum_{k=1}^{\infty}k\sum_{|\alpha|=k}|a_{\alpha}|^2\;r^{2\alpha}\right)\leq 1,
\end{align}
for ${n\bf r}\leq \sqrt{5} - 2 \approx 0.236068$, where $z=(z_1,\ldots,z_n)\in \mathbb{P}\Delta(0;1/n)$ and $r=(r_1,r_2,\ldots,r_n)$ such that ${\bf r}=||z||_{\infty}$ and the constants $n{\bf r_{0}} = \sqrt{5} - 2$ and $p = 2(\sqrt{5} - 1)$ are sharp.
\end{theo}
\section{{\bf Key lemmas and proofs}}\label{Sec-3}
We now recall here some lemmas for functions of single complex variable.
\begin{lemA}\cite{Kayumov-Ponnusamy-CMFT-2017,Kayumov-Ponnusamy-CRMASP-2018} Let $|b_0|<1$ and $g(z)=\sum_{k=1}^{\infty} b_k z^k$ be analytic and satisfies $|g(z)|<1$ in $\mathbb{D}$.
\begin{enumerate}
	\item[\emph{(a).}] If  $0<r\leq \frac{1}{\sqrt{2}}$, then $\sum_{k=1}^{\infty} k|b_k|^2 r^{2k}\leq r^2\frac{(1-|b_0|^2)^2}{(1-|b_0|^2\;r^2)^2}.$ 
	\item[\emph{(b).}] If $0<R\leq 1$, then $\sum_{k=1}^{\infty} |b_k|^2 R^k\leq R\frac{(1-|b_0|^2)^2}{1-|b_0|^2R}.$
\end{enumerate}
\end{lemA}
The following key lemma is established as a multidimensional analogue of parts (a) and (b) of Lemma A, serving as the foundation for the proof of our main results.
\begin{lem}\label{lem1} Let $f$ be a holomorphic function in the polydisk $\mathbb{P}\Delta(0;1_n)$ such that $|f(z)|\le 1$ for all $z\in \mathbb{P}\Delta(0;1_n)$ and $f(z)=\sum_{|\alpha|=0}^{\infty}a_{\alpha}z^{\alpha}$ for all $z\in \mathbb{P}\Delta(0;1_n)$. Then, for $r=(r_1,r_2,\ldots,r_n)$ and ${\bf{r}}=||r||_{\infty}$, we have the estimates
\begin{enumerate}
	\item[\emph{(a).}] \begin{align}
		\label{ms0} \sum_{k=1}^{\infty}k\sum_{|\alpha|=k}|a_{\alpha}|^2\;{\bf r}^{2|\alpha|}\leq {\bf r}^2\frac{(1-|a_0|^2)^2}{(1-|a_0|^2\;{\bf r}^2)^2},\; \mbox{for}\; 0<{\bf r}\leq 1/\sqrt{2}.
	\end{align}
	\item[\emph{(b).}] \begin{align}
		\label{ams0}\sum_{k=1}^{\infty}\sum_{|\alpha|=k}||a_{\alpha}|^2\;{\bf r}^{|\alpha|}\leq {\bf r}\frac{(1-|a_0|^2)^2}{1-|a_0|^2\;{\bf r}},\; \mbox{for}\; 0<{\bf r}<1.
	\end{align}
	\item[\emph{(c).}] \begin{align*} \sum_{k=1}^{\infty}\sum_{|\alpha|=k} |a_{\alpha}|r^{\alpha}\leq \sum_{k=1}^{\infty}\sum_{|\alpha|=k} |a_{\alpha}|{\bf r}^{\alpha}\leq \left\{\begin{array}{clcr}
			\sqrt{n}{\bf r}\dfrac{1-|a_0|^2}{1-n|a_0|{\bf r}}&\text{for}\;\;|a_0|\geq {\bf r}\vspace{2mm}\\
			\smallskip
			\sqrt{n}{\bf r}\dfrac{\sqrt{1-|a_0|^2}}{\sqrt{1-n{\bf r}^2}}&\mbox{for}\;|a_0|<{\bf r}.\end{array}\right.
	\end{align*}
\end{enumerate}
The inequalities \eqref{ms0} and \eqref{ams0} are sharp.
\end{lem}
\begin{proof} The function $f$ has the following series expansion in terms of the homogeneous polynomials
\begin{align}\label{Eq-3.2}
	f(z)=a_0+P_{1}(z)+P_{2}(z)+\ldots,
\end{align}
where the term $P_k(z)$ is a homogeneous polynomial of degree {$k$}.\vspace{1.2mm}

For any $z\in\mathbb{P}\Delta(0;\mathbf{1}_n)$, the theorem holds trivially if $z=\mathbf{0}_n$; otherwise (if $z \neq \mathbf{0}_n$), we proceed to define
\begin{align}\label{ms1}g(t)=f\left(\frac{z}{||z||_{\infty}}t\right)=a_0+P_{1}\left(\frac{z}{||z||_{\infty}}\right)t+P_{2}\left(\frac{z}{||z||_{\infty}}\right)t^2+\ldots,
\end{align}
where $t\in\mathbb{C}$ such that $|t|<1$.\vspace{1.2mm}

\noindent{\bf Proof of (a).} For a fixed $z \in \mathbb{P}\Delta(0;1_n)$, the function $g(t)$ is analytic in the disk $|t|<1$. Since $|f(z)| \leq 1$, it follows from Eq. $\eqref{ms1}$ that $|g(t)| \leq 1$ in $|t|<1$. Thus, we have
\begin{align}
	\label{ms2} \sum_{k=1}^{\infty} k\left|P_{k}\left(\frac{z}{||z||_{\infty}}\right)\right|^2\;|t|^{2k}\leq |t|^2\frac{(1-|a_0|^2)^2}{(1-|a_0|^2\;|t|^2)^2}\; \mbox{for}\; 0<|t|<1.
\end{align}
Let $z=(z_1,z_2,\ldots,z_n)\in \mathbb{P}\Delta(0;1_n)$ such that $|z_i|=r_i$ for $i=1,2,\ldots,n$ and ${\bf{r}}=||r||_{\infty}=||z||_{\infty}$, where $r=(r_1,r_2,\ldots,r_n)$.  Then from (\ref{ms2}), we obtain
\begin{align}
\label{ms3} \sum_{k=1}^{\infty} k\left|P_{k}\left(\frac{z}{{\bf r}}\right)\right|^2\;{\bf r}^{2k}\leq {\bf r}^2\frac{(1-|a_0|^2)^2}{(1-|a_0|^2\;{\bf r}^2)^2},\; \mbox{for}\; 0<{\bf r}<1,
\end{align}
where
\begin{align*}
	P_{k}\left(\frac{z}{{\bf r}}\right)= \sum_{|\alpha|=k}a_{\alpha}\left(\frac{z}{{\bf r}}\right)^{\alpha}.
\end{align*}
We set $z_j=r_je^{i\theta_j}\;(0\leq\theta_j\leq 2\pi)$, where $0<r_j<1$, $j=1,2,\ldots,n$ and consider the integration of $\left|P_{k}\left(\frac{r_1e^{i\theta_1}}{{\bf r}},\ldots,\frac{r_ne^{i\theta_n}}{{\bf r}}\right)\right|^2$. Noting that $
\int_0^{2\pi} e^{i(\beta-\gamma)\theta}d\theta=2\pi \delta_{\beta,\gamma}$ Kronecker's symbol, we have
\begin{align*}
	\sum_{k=1}^{\infty}|a_{\alpha}|^2r^{2\alpha}\; \frac{1}{{\bf r}^{2|\alpha|}}&=\frac{1}{(2\pi)^n} \int_{0}^{2\pi} \cdots \int_{0}^{2\pi} \left|P_{k}\left(\frac{r_1e^{i\theta_1}}{{\bf r}},\ldots,\frac{r_ne^{i\theta_n}}{{\bf r}}\right)\right|^2 \, d\theta_1 \cdots d\theta_n\\&\leq
	\max_{C^n(0,r)}\left|P_{k}\left(\frac{r_1e^{i\theta_1}}{{\bf r}},\ldots,\frac{r_ne^{i\theta_n}}{{\bf r}}\right)\right|^2,
\end{align*}
\textit{i.e.,} 
\begin{align}
	\label{ms4} \sum_{|\alpha|=k}|a_{\alpha}|^2r^{2\alpha}\;\frac{1}{{\bf r}^{2|\alpha|}}\leq \max_{C^n(0,r)}\left|P_{k}\left(\frac{r_1e^{i\theta_1}}{{\bf r}},\ldots,\frac{r_ne^{i\theta_n}}{{\bf r}}\right)\right|^2,
\end{align}
where $k\geq 1$ and $r=(r_1,r_2,\ldots,r_n)$. On the other hand, from \eqref{ms3}, we obtain
\begin{align}
	\label{ms5} \sum_{k=1}^{\infty} k\max_{C^n(0,r)}\left|P_{k}\left(\frac{z}{{\bf r}}\right)\right|^2\;{\bf r}^{2k}\leq {\bf r}^2\frac{(1-|a_0|^2)^2}{(1-|a_0|^2\;{\bf r}^2)^2},\; \mbox{for}\;0<{\bf r}<1.
\end{align}
In view of (\ref{ms4}) to (\ref{ms5}), we obtain
\begin{align}\label{ms6} \sum_{k=1}^{\infty}k\sum_{|\alpha|=k}|a_{\alpha}|^2r^{2\alpha}\;||a_{\alpha}r^{\alpha}||_1^2\leq {\bf r}^2\frac{(1-|a_0|^2)^2}{(1-|a_0|^2\;{\bf r}^2)^2},\; \mbox{for}\;0<{\bf r}<1.
\end{align}
Letting $r_i\to {\bf r}$ for $i=1,2,\ldots,n$, we obtain from (\ref{ms6}) that 
\begin{align*}
	\sum_{k=1}^{\infty} k \sum_{|\alpha|=k}|a_{\alpha}|^2\;{\bf r}^{2|\alpha|}\leq {\bf r}^2\frac{(1-|a_0|^2)^2}{(1-|a_0|^2\;{\bf r}^2)^2},\; \mbox{for}\;0<{\bf r}<1.
\end{align*}
For $0 < \mathbf{r} \leq 1/\sqrt{2}$, the sequence $\{|\alpha|\mathbf{r}^{2|\alpha|}\}$ is non-increasing for all $|\alpha| \geq 1$, which is equivalent to the condition $(k+1)\mathbf{r}^{2(\alpha+1)}\leq k\mathbf{r}^{2\alpha}$.
\textit{i.e.,} if 
\begin{align*}
	{\bf r}\leq \sqrt{\frac{k}{k+1}}\;\;k=1,2,\ldots
\end{align*} 
Since the last bound is smallest for $k=1$, the condition is equivalent to $\mathbf{r \leq \frac{1}{\sqrt{2}}}$.\vspace{1.2mm}

To show the inequality \eqref{ms0} is sharp, we consider the function
\begin{align}\label{Eq-3.9}
	f(z)=\frac{a-\frac{1}{n}(z_1+z_2+\ldots+z_n)}{1-\frac{a}{n}(z_1+z_2+\ldots+z_n)},\; \mbox{where}\; a=|a_0|\neq 0.
\end{align}
The function $f$ is easily verified to be holomorphic in the polydisk $\mathbb{P}\Delta(0;\mathbf{1}_n)$. Since $\left|\frac{a}{n}(z_1+z_2+\ldots+z_n)\right|<1,$ we se that $f$ has the following series expansion
\[f(z)=\sum_{|\alpha|=0}^{\infty} a_{\alpha}z^{\alpha}=a-(1-a^2)\sum_{k=1}^{\infty} a^{k-1}\left(\frac{z_1+z_2+\ldots+z_n}{n}\right)^k\]
for all $z\in\mathbb{P}\Delta(0;1_n)$. For the point $z=(r,r,\ldots,r)$, we find that
\begin{align}\label{Eq-3.11}
	f(z)=\sum_{|\alpha|=0}^{\infty} a_{\alpha}r^{\alpha}=a-(1-a^2)\sum_{k=1}^{\infty} a^{k-1}{\bf r}^k.
\end{align}
Thus, we see that
\begin{align*}
	\sum_{k=1}^{\infty}k \sum_{|\alpha|=k}|a_{\alpha}|^2\;r^{2|\alpha|}=\frac{(1-a^2)}{a^2}(a{\bf r})^2\sum_{k=0}^{\infty} (a^2{\bf r}^2)^{k}={\bf r}^2\frac{(1-a^2)^2}{(1-a^2{\bf r}^2)^2},
\end{align*}
which shows that the inequality (\ref{ms0}) is sharp. \vspace{1.2mm}

\noindent{\bf Proof of (b).} The function $f$ has the series expansion \eqref{Eq-3.2}. Let $g$ be defined as in \eqref{ms1}. Because $|f(z)| \leq 1$, we deduce from \eqref{ms1} that $|g(t)| \leq 1$ for $|t|<1$. Applying Lemma B yields that
\begin{align}\label{ams2} \sum_{k=1}^{\infty} \left|P_{k}\left(\frac{z}{||z||_{\infty}}\right)\right|^2\;|t|^{k}\leq |t|\frac{(1-|a_0|^2)^2}{1-|a_0|^2\;|t|},\; \mbox{for}\;0<|t|<1.
\end{align}
If $0 < \mathbf{r} < 1$, then Eq. (\ref{ams2}) yields
\begin{align}\label{ams3} \sum_{k=1}^{\infty} \left|P_{k}\left(\frac{z}{||z||_{\infty}}\right)\right|^2\;{\bf r}^{k}\leq {\bf r}\frac{(1-|a_0|^2)^2}{1-|a_0|^2\;{\bf r}}\; \mbox{for}\;0<{\bf r}<1.
\end{align}
 Using similar arguments as those applied in part (a), we can prove that
\begin{align}
	\label{ams4} \sum_{|\alpha|=k}|a_{\alpha}|^2r^{2\alpha}\;\frac{1}{{\bf r}^{2|\alpha|}}\leq \max_{C^n(0,r)}\left|P_{k}\left(\frac{r_1e^{i\theta_1}}{{\bf r}},\ldots,\frac{r_ne^{i\theta_n}}{{\bf r}}\right)\right|^2.
\end{align}
Furthermore, it follows from \eqref{ams3} that
\begin{align}
	\label{ams5} \sum_{k=1}^{\infty} \max_{C^n(0,r)}\left|P_{k}(z)\right|^2\;{\bf r}^{k}\leq {\bf r}\frac{(1-|a_0|^2)^2}{1-|a_0|^2\;{\bf r}},\; \mbox{for}\; 0<{\bf r}<1.
\end{align}
By substituting (\ref{ams4}) into (\ref{ams5}), we obtain
\begin{align}	\label{ams6} \sum_{k=1}^{\infty} \sum_{|\alpha|=k}|a_{\alpha}|^2r^{2\alpha}\;\frac{1}{{\bf r}^{|\alpha|}}\leq {\bf r}\frac{(1-|a_0|^2)^2}{1-|a_0|^2\;{\bf r}},\; \mbox{for}\; 0<{\bf r}<1.
\end{align}
 Taking the limit $r_i \to \mathbf{r}$, Eq. (\ref{ams6}) yields that desired inequality \eqref{ams0} as
\begin{align*}
	\sum_{k=1}^{\infty}\sum_{|\alpha|=k}|a_{\alpha}|^2\;{\bf r}^{|\alpha|}\leq {\bf r}\frac{(1-|a_0|^2)^2}{1-|a_0|^2\;{\bf r}}\; \mbox{for}\;0<{\bf r}<1.
\end{align*}
The function $f$ defined by (\ref{Eq-3.9}) is considered to show that this inequality is sharp. In fact, we see that
\begin{align*}
	\sum_{k=1}^{\infty} \sum_{|\alpha|=k}|a_{\alpha}|^2\;r^{|\alpha|}=\frac{(1-a^2)}{a^2}(a\sqrt{{\bf r}})^2\sum_{k=0}^{\infty}(a^2{\bf r})^{k}=r\frac{(1-a^2)^2}{1-a^2{\bf r}},
\end{align*}
which shows that the inequality (\ref{ams0}) is sharp.\vspace{1.2mm}

\noindent{\bf Proof of (c).} Choose any $\rho>1$ such that $\rho {\bf r}\leq 1$ and $n\rho^{-1}{\bf r}<1$. Then, using the classical Cauchy-Schwarz inequality and Eq. (\ref{ams0}) (with $\rho \mathbf{r}$ substituted for $\mathbf{r}$), we have the estimate
\begin{align}
	\label{mmss5} \sum_{k=1}^{\infty}\sum_{|\alpha|=k} |a_{\alpha}|r^{\alpha}\leq \sum_{k=1}^{\infty}\sum_{|\alpha|=k} |a_{\alpha}|{\bf r}^{\alpha}&\leq \sqrt{\sum_{k=1}^{\infty}\sum_{|\alpha|=k} |a_{\alpha}|^2\rho^{|\alpha|}{\bf r}^{|\alpha|}} \sqrt{\sum_{k=1}^{\infty}\sum_{|\alpha|=k} \rho^{-|\alpha|}{\bf r}^{|\alpha|}}\\&\leq
	\sqrt{\rho {\bf r}\frac{(1-|a_0|^2)^2}{1-|a_0|^2\rho {\bf r}}}\sqrt{\sum_{k=1}^{\infty}({\rho^{-1} \bf r})^{k}\sum_{|\alpha|=k} 1}\nonumber\\&\leq
	\sqrt{\rho {\bf r}\frac{(1-|a_0|^2)^2}{1-|a_0|^2\rho {\bf r}}} \sqrt{\sum_{k=1}^{\infty}({n\rho^{-1}\bf r})^{k}}\nonumber\\&\leq
	\sqrt{\rho {\bf r}\frac{(1-|a_0|^2)^2}{1-|a_0|^2\rho {\bf r}}} \sqrt{\frac{n\rho^{-1}{\bf r}}{1-n\rho^{-1}{\bf r}}}\nonumber\\&=
	\frac{\sqrt{n} {\bf r}(1-|a_0|^2)}{\sqrt{1-|a_0|^2\rho {\bf r}}\sqrt{1-n\rho^{-1}{\bf r}}}.\nonumber
\end{align}
We independently consider the cases where $|a_0| \geq \mathbf{r}$ (setting $\rho=1/|a_0|$) and $|a_0| < \mathbf{r}$ (setting $\rho=1/\mathbf{r}$), noting that $1-n|a_0|\mathbf{r} < 1-|a_0|\mathbf{r}$, which allows us to obtain the following inequality from Eq. (\ref{mmss5})
\begin{align*} \sum_{k=1}^{\infty}\sum_{|\alpha|=k} |a_{\alpha}|r^{\alpha}\leq \sum_{k=1}^{\infty}\sum_{|\alpha|=k} |a_{\alpha}|{\bf r}^{\alpha}\leq \left\{\begin{array}{clcr}
		\sqrt{n}{\bf r}\dfrac{1-|a_0|^2}{1-n|a_0|{\bf r}},&\text{for}\;\;|a_0|\geq {\bf r}\vspace{2mm}\\
          \medskip
		\sqrt{n}{\bf r}\dfrac{\sqrt{1-|a_0|^2}}{\sqrt{1-n{\bf r}^2}},&\text{for}\;|a_0|<{\bf r}.\end{array}\right.
\end{align*}
This completes the proof.
\end{proof}
We have the following lemma which is a special case of \cite[Theorem 2.2]{Chen-Hamada-Ponnusamy-Vijayakumar-JAM-2024}.
\begin{lem}\label{Lem4}Let $f$ be holomorphic in the polydisk $\mathbb{P}\Delta(0;1_n)$ such that $|f(z)|\le 1$ for all $z\in \mathbb{P}\Delta(0;1_n)$. Then for all $z\in \mathbb{P}\Delta(0;1_n)$, we have
\[|f(z)|\leq \frac{|f(0)|+||z||_{\infty}}{1+|f(0)|||z||_{\infty}}.\]
\end{lem}
\section{{\bf Proof of the main results}}\label{Sec-4}
In this section, we present the proof of the main results.
\begin{proof}[\bf Proof of Theorem \ref{Th-2.1}] Assuming the given condition, the function $f(z)=a_0+\sum_{|\alpha|=1}^{\infty} a_{\alpha} z^{\alpha}$ is holomorphic and bounded ($|f(z)|\leq 1$) in the polydisk $\mathbb{P}\Delta(0;\mathbf{1}_n)$, where we define the maximum norm of $z=(z_1,\ldots,z_n)$ as $\mathbf{r}=||z||_{\infty}$. Note that 
\begin{align}
	\label{AM2} \sum_{|\alpha|=0}^{\infty} |a_{\alpha}| |z|^{\alpha}\leq \sum_{|\alpha|=0}^{\infty}|a_{\alpha}|\; ||z||_{\infty}^{|\alpha|}=\sum_{|\alpha|=0}^{\infty} |a_{\alpha}|\; {\bf r}^{|\alpha|}=
	|a_0|+\sum_{k=1}^{\infty}\sum_{|\alpha|=k} |a_{\alpha}|\;{\bf r}^{|\alpha|}.
\end{align}
Applying Lemma \ref{lem1} $(c)$, we find that
\begin{align}\label{AM3} \sum_{k=1}^{\infty}\sum_{|\alpha|=k} |a_{\alpha}|r^{\alpha}\leq \sum_{k=1}^{\infty}\sum_{|\alpha|=k} |a_{\alpha}|\;{\bf r}^{|\alpha|}\leq \left\{\begin{array}{clcr}
		A({\bf r}):= \sqrt{n}{\bf r}\dfrac{1-|a_0|^2}{1-n|a_0|{\bf r}},&\text{for}\;\;|a_0|\geq {\bf r}\\		
          \medskip
		B({\bf r}):= \sqrt{n}{\bf r}\dfrac{\sqrt{1-|a_0|^2}}{\sqrt{1-n{\bf r}^2}},&\text{for}\;|a_0|<{\bf r}.\end{array}\right.
\end{align}
By Lemma \ref{lem1} $(a)$, we see that
\begin{align}\label{AM4} \sum_{k=1}^{\infty}k\sum_{|\alpha|=k}|a_{\alpha}|^2\;r^{2\alpha}&\leq\sum_{k=1}^{\infty}k\sum_{|\alpha|=k}|a_{\alpha}|^2\;{\bf r}^{2|\alpha|}\leq {\bf r}^2\frac{(1-|a_0|^2)^2}{(1-|a_0|^2\;{\bf r}^2)^2},
\end{align}
for $\;0<{\bf r}\leq 1/\sqrt{2}$. Since $|a_0|{\bf r}^2\leq |a_0|{n^2\bf r}^2$, it follows from (\ref{AM4}) that
\begin{align}
	\label{AM5} {\bf r}^2\frac{(1-|a_0|^2)^2}{(1-|a_0|^2\;{\bf r}^2)^2}\leq {n^2\bf r}^2\frac{(1-|a_0|^2)^2}{(1-|a_0|^2\;{n^2\bf r}^2)^2}.
\end{align}
In view of (\ref{AM2}), an easy computation shows that
\begin{align}
	\label{AM6} {\mathcal{A}_f(r)}&=\sum_{|\alpha|=0}^{\infty} |a_{\alpha}| |z|^{\alpha}+\frac{16}{9}\sum_{k=1}^{\infty}k\sum_{|\alpha|=k}|a_{\alpha}|^2\;r^{2\alpha}+\lambda \left(\sum_{k=1}^{\infty}k\sum_{|\alpha|=k}|a_{\alpha}|^2\;r^{2\alpha}\right)^2\\&\leq
	|a_0|+\sum_{k=1}^{\infty}\sum_{|\alpha|=k} |a_{\alpha}|\;{\bf r}^{|\alpha|}+\frac{16}{9}\sum_{k=1}^{\infty}k\sum_{|\alpha|=k}|a_{\alpha}|^2\;{\bf r}^{2|\alpha|}\nonumber+\lambda\left(\sum_{k=1}^{\infty}k\sum_{|\alpha|=k}|a_{\alpha}|^2\;{\bf r}^{2|\alpha|}\right)^2\nonumber\\&:={\mathcal{A}_f({\bf r})}.\nonumber
\end{align}
Since ${\mathcal{A}_f({\bf r})}$ is an increasing function of ${\bf r}$, we have
${\mathcal{A}_f({\bf r})}\leq {\mathcal{A}_f\left(\frac{1}{3n}\right)}$, for $0\leq {\bf r}\leq 1/3n$
and thus, it suffices to prove the inequality (\ref{2.1}) for ${\bf r}=1/3n$.\vspace{2mm}

To complete the proof, we discuss the following two cases:\vspace{1.2mm}

\noindent{\bf Case 1.} Let $a=|a_0|\geq 1/3n$. In this case, for $|a_0| \geq \mathbf{r}$, we see from Eq. (\ref{AM3}) that
\begin{align}
	\label{AM7} A({\bf r}):= \sqrt{n}{\bf r}\left(\frac{1-|a_0|^2}{1-n|a_0|{\bf r}}\right)\leq n{\bf r}\left(\frac{1-|a_0|^2}{1-n|a_0|{\bf r}}\right).
\end{align}
Substituting (\ref{AM5}) and (\ref{AM7}) (with $\mathbf{r}=1/3n$) into Equation (\ref{AM6}), a straightforward computation shows that
\begin{align}
\label{AM8}{\mathcal{A}_f({\bf r})}&\leq a+A\left(\frac{1}{3n}\right)+\frac{16}{9}\left({\bf r}^2\frac{(1-|a_0|^2)^2}{(1-|a_0|^2\;{\bf r}^2)^2}\right)+\lambda\left({\bf r}^2\frac{(1-|a_0|^2)^2}{(1-|a_0|^2\;{\bf r}^2)^2}\right)^2\\&\leq
	a+A\left(\frac{1}{3n}\right)+\frac{16}{9}\left({n^2\bf r}^2\frac{(1-a^2)^2}{(1-a^2\;{n^2\bf r}^2)^2}\right)+\lambda\left({n^2\bf r}^2\frac{(1-a^2)^2}{(1-a^2\;{n^2\bf r}^2)^2}\right)^2\nonumber\\&\leq
	a+\frac{1-a^2}{3-a}+16\frac{(1-a^2)^2}{(9-a^2)^2}+81\lambda \frac{(1-a^2)^4}{(9-a^2)^4}=
	1-\frac{(1-a)^3}{(9-a^2)^4}\Phi(a),\nonumber
\end{align}
where
\begin{align}
	\label{AM9} \Phi(t)&:= 3078 + 1944t - 522t^{2} - 432t^{3} + 2t^{4} + 24t^{5} + 2t^{6}\\&
	+ \lambda_1\left( -81 - 243t - 162t^{2} + 162t^{3} + 243t^{4} + 81t^{5} \right).\nonumber
\end{align}
Let us choose $s\in [0,1]$ such that $\Phi'(s)=0$. Subsequently, (\ref{AM9}) yields $\lambda_1$ of Theorem C. By substituting $\lambda_1$ into (\ref{AM9}) and performing the necessary simplification, it is easily deduced that
\begin{align}
	\label{AM11} \Phi(s) = \frac{2s^{2} - 9}{3 - 5s}\, \psi_1(s).
\end{align}
One can verify that the function $\psi(s)$ has exactly one positive root in the interval $[1/3, 1]$, denoted $s_0 \approx 0.567284$.
Furthermore, from (\ref{AM11}), we observe that $\Phi(1/3)>0$ and $\Phi(1)>0$. Since $\Phi(s)$ is non-negative on the boundary, we conclude that $\Phi(s)\geq 0$ throughout the interval $[1/3, 1]$. This shows that ${\mathcal{A}_f(\mathbf{r})} \leq 1$ for $|a_0| \in [1/3, 1]$ when $\mathbf{r} = 1/3n$.\vspace{1.2mm}

\noindent{\bf Case 2.} Let $a=|a_0|<1/3n$. Since $n{\bf r}^2\leq {n^2\bf r}^2$, for $|a_0|<{\bf r}$, we see from (\ref{AM3}) that
\begin{align}
	\label{AM12} B({\bf r})=: \sqrt{n}{\bf r}\frac{\sqrt{1-|a_0|^2}}{\sqrt{1-n{\bf r}^2}}\leq n{\bf r}\frac{\sqrt{1-|a_0|^2}}{\sqrt{1-n^2{\bf r}^2}}.
\end{align}
In view of (\ref{AM5}) and (\ref{AM12}) (with $\mathbf{r}=1/3n$), when applied to (\ref{AM6}), a simple computation leads to the following
\begin{align*}
	\label{AM13}{\mathcal{A}_f({\bf r})}&\leq a+B\left(\frac{1}{3n}\right)+\frac{16}{9}\left({\bf r}^2\frac{(1-|a_0|^2)^2}{(1-|a_0|^2\;{\bf r}^2)^2}\right)+\lambda\left({\bf r}^2\frac{(1-|a_0|^2)^2}{(1-|a_0|^2\;{\bf r}^2)^2}\right)^2\\&\leq
	a+B\left(\frac{1}{3n}\right)+\frac{16}{9}\left({n^2\bf r}^2\frac{(1-a^2)^2}{(1-a^2\;{n^2\bf r}^2)^2}\right)+\lambda\left({n^2\bf r}^2\frac{(1-a^2)^2}{(1-a^2\;{n^2\bf r}^2)^2}\right)^2\nonumber\\&\leq
	a+\frac{\sqrt{1-a^2}}{\sqrt{8}}+16\frac{(1-a^2)^2}{(9-a^2)^2}+81\lambda \frac{(1-a^2)^4}{(9-a^2)^4}:=
	\Psi(a).
\end{align*}
Routine calculations show that the final expression for $\Psi(s)$ is monotonically increasing with respect to $s$. Over the interval $[0, 1/3)$, the maximum value is therefore achieved at the boundary point $s = 1/3$. Since the maximum $\Psi(1/3) \leq 0.98$, the inequality (\ref{2.1}) is satisfied for $|a_{0}| \in [0, 1/3)$. This completes the proof that ${\mathcal{A}_f({\bf r})} \leq 1$ holds for $|a_{0}| \in [0, 1/3)$ and $\mathbf{r} \leq 1/3n$.\vspace{1.2mm}

To prove that the constant $\lambda_2$ is sharp, we let $a\in [0, 1)$ and consider the function $f_a$ as given by \eqref{Eq-2.2}. For the point $z=(r,r,\ldots,r)$, we find that
\begin{align*}
	f_a(z)=a-(1-a^2)\sum_{k=1}^{\infty} a^{k-1}(n{\bf r})^k
\end{align*}
and hence, we have the estimates
	\begin{align*}
		\begin{cases}
			\displaystyle\sum_{|\alpha|=0}|a_{\alpha}||z|^{\alpha}=a+(1-a^2)\displaystyle\sum_{k=1}^{\infty} a^{k-1}(n{\bf r})^k=a+(1-a^2)\dfrac{n{\bf r}}{1-na{\bf r}},\vspace{2mm}\\
			\displaystyle\sum_{k=1}^{\infty} k\sum_{|\alpha|=k}|a_{\alpha}|^2\;r^{2|\alpha|}=\displaystyle\frac{(1-a^2)}{a^2}(an{\bf r})^2\sum_{k=0}^{\infty}\binom{^{k+1}}{_1} (a^2{\bf r}^2)^{k}=(n{\bf r})^2\dfrac{(1-a^2)^2}{(1-a^2n^2{\bf r}^2)^2}.
		\end{cases}
\end{align*}
Therefore, in view of the above estimates, we see that
\begin{align*}
{\mathcal{A}_f(r)}&=\sum_{|\alpha|=0}^{\infty}|a_{\alpha}| |z|^{\alpha}+\frac{16}{9}\sum_{k=1}^{\infty}k\sum_{|\alpha|=k}|a_{\alpha}|^2\;r^{2\alpha}+\lambda \left(\sum_{k=1}^{\infty}k\sum_{|\alpha|=k}|a_{\alpha}|^2\;r^{2\alpha}\right)^2\\&=
	a+(1-a^2)\frac{n{\bf r}}{1-na{\bf r}}+\frac{16}{9}\left((n{\bf r})^2\frac{(1-a^2)^2}{(1-a^2n^2{\bf r}^2)^2}\right)+\lambda \left((n{\bf r})^2\frac{(1-a^2)^2}{(1-a^2n^2{\bf r}^2)^2}\right)^2.
\end{align*}
Consequently, we have
\[{\mathcal{A}_f\left(\frac{1}{3n}\right)}=a + \frac{1 - a^{2}}{3 - a}+ \frac{16(1 - a^{2})^{2}}{(9 - a^{2})^{2}}
+ \frac{81\lambda (1 - a^{2})^{4}}{(9 - a^{2})^{4}}.\]
Choosing $a$ as the positive root $s_0$ of the equation $\psi(s)=0$, we consequently see that $\mathcal{A}_f\left(\frac{1}{3n}\right)=1$.
Moreover,
\begin{align*}
{\mathcal{A}_f\left(\frac{1}{3n}\right)}&=a + \frac{1 - a^{2}}{3 - a}
	+ \frac{16(1 - a^{2})^{2}}{(9 - a^{2})^{2}}
	+ \frac{81\lambda_1 (1 - a^{2})^{4}}{(9 - a^{2})^{4}}\\&=
	a 
	+ \frac{1 - a^{2}}{3 - a}
	+ \frac{16(1 - a^{2})^{2}}{(9 - a^{2})^{2}}
	+ \frac{81\lambda (1 - a^{2})^{4}}{(9 - a^{2})^{4}}
	+ \frac{81(\lambda_{3} - \lambda)(1 - a^{2})^{4}}{(9 - a^{2})^{4}}\\&:={\mathcal{A}_{f,1}\left(\frac{1}{3n}\right)}
	.
\end{align*}
 Thus, we see that the expression
\begin{align*}
\mathcal{A}_{f,1}\left(\frac{1}{3n}\right)=1+ (\lambda_{3} - \lambda)\frac{81(1 - a^{2})^{4}}{(9 - a^{2})^{4}},
\end{align*}
is bigger than $1$ in the case when $\lambda_3>\lambda$ which shows that the constant $\lambda$ is best possible. This completes the proof.
\end{proof}

\begin{proof}[\bf Proof of Theorem \ref{Th-2.2}] 
By the given condition $f(z)=a_0+\sum_{|\alpha|=1}^{\infty} a_{\alpha} z^{\alpha}$
is holomorphic in the polydisk $\mathbb{P}\Delta(0;1_n)$ such that $|f(z)|\leq 1$ in $\mathbb{P}\Delta(0;1_n)$. Let us take $z=(z_1,\ldots,z_n)\in \mathbb{P}\Delta(0;1_n)$ such that ${\bf r}=||z||_{\infty}$. By Lemma \ref{Lem4}, we have
\begin{align}
	\label{MA1} |f(z)| \leq \frac{||z||_{\infty}+|a_0|}{1+|a_0|||z||_{\infty}}=\frac{{\bf r}+|a_0|}{1+|a_0|{\bf r}}\le \frac{n{\bf r}+|a_0|}{1+|a_0|n{\bf r}}.
\end{align}
Using (\ref{AM2}) and (\ref{MA1}), we see that
\begin{align}
	\label{MA2} &{\mathcal{B}_f(z,r)}\nonumber\\&=|f(z)|^2+\sum_{|\alpha|=1}^{\infty} |a_{\alpha}| |z|^{\alpha}+\frac{16}{9}\sum_{k=1}^{\infty}k\sum_{|\alpha|=k}|a_{\alpha}|^2\;r^{2\alpha}+\lambda \left(\sum_{k=1}^{\infty}k\sum_{|\alpha|=k}|a_{\alpha}|^2\;r^{2\alpha}\right)^2\nonumber\\&\leq
	\left(\frac{n{\bf r}+|a_0|}{1+|a_0|n{\bf r}}\right)^2+\sum_{k=1}^{\infty}\sum_{|\alpha|=k} |a_{\alpha}|\;{\bf r}^{|\alpha|}+\frac{16}{9}\sum_{k=1}^{\infty}k\sum_{|\alpha|=k}|a_{\alpha}|^2\;{\bf r}^{2|\alpha|}\nonumber\\&\quad+\lambda\left(\sum_{k=1}^{\infty}k\sum_{|\alpha|=k}|a_{\alpha}|^2\;{\bf r}^{2|\alpha|}\right)^2={\mathcal{B}_f(z,{\bf r})}.
\end{align}
Since $\mathcal{B}_f(z,\mathbf{r})$ increases with $\mathbf{r}$, the inequality $\mathcal{B}_f(z,\mathbf{r})\leq \mathcal{B}_f(z, 1/3n)$ holds for $0\leq \mathbf{r}\leq 1/3n$, meaning we only need to consider $\mathbf{r}=1/3n$ to establish inequality (\ref{2.2}).\vspace{1.2mm}

We consider following two cases:

\medskip
{\bf Case 1.} Let $a=|a_0|\geq 1/3n$.  Using (\ref{AM5}) and (\ref{AM7}) (with ${\bf r}=1/3n$) to (\ref{MA2}), we have
\begin{align*}&{\mathcal{B}_f(z,{\bf r})}\\&\leq \left(\frac{n{\bf r}+|a_0|}{1+|a_0|n{\bf r}}\right)^2+A({\bf r})+\frac{16}{9}\left({\bf r}^2\frac{(1-|a_0|^2)^2}{(1-|a_0|^2\;{\bf r}^2)^2}\right)+\lambda\left({\bf r}^2\frac{(1-|a_0|^2)^2}{(1-|a_0|^2\;{\bf r}^2)^2}\right)^2\nonumber\\&\leq
	\left(\frac{n{\bf r}+|a_0|}{1+|a_0|n{\bf r}}\right)^2+A({\bf r})+\frac{16}{9}\left({n^2\bf r}^2\frac{(1-a^2)^2}{(1-a^2\;{n^2\bf r}^2)^2}\right)+\lambda\left({n^2\bf r}^2\frac{(1-a^2)^2}{(1-a^2\;{n^2\bf r}^2)^2}\right)^2\nonumber\\&\leq
	\left(\frac{1+3a}{3+a}\right)^2+\frac{1-a^2}{3-a}+16\frac{(1-a^2)^2}{(9-a^2)^2}+81\lambda \frac{(1-a^2)^4}{(9-a^2)^4}\nonumber\\&=
	1 
	-(1 - a)^{3}(1 +a)\bigg(\frac{a+29}{(9-a^2)^2}
	-81\lambda\frac{(1 - a^{2})(1 + a)^{2}}{(9 - a^{2})^{4}}\bigg)\nonumber\\&=
	1 - \frac{(1 - a)^{3}(1 + a)}{(9 - a^{2})^{4}}
	\, \Phi(a),\nonumber
\end{align*}
where
\begin{align}\label{MA4}\Phi(t)&= 2349 + 81t - 522t^{2} - 18t^{3} + 29t^{4} + t^{5}\\&+ \lambda_2(-81 - 162t + 162t^{3} + 81t^{4}).\nonumber
\end{align}
Let us choose $s\in [0,1]$ such that $\Phi'(s)=0$. Then, from $(\ref{MA4})$, we get the value of $\lambda_2$. Substituting this value of $\lambda_2$ back into $(\ref{MA4})$ and simplifying leads to
\begin{align}
	\label{MA6} \Phi(s) =\frac{9 - s^{2}}{2(2s - 1)}\,(-513 + 910s + 80s^{2} + 2s^{3} + s^{4}).
\end{align}
Let $\psi_2(s) = -513 + 910s + 80s^{2} + 2s^{3} + s^{4}$. One can verify that $\psi(s)$ possesses exactly one positive root in the interval $[1/3, 1]$, approximately $s_0 \approx 0.537869$.
Furthermore, from (\ref{MA6}), we observe that the boundary values $\Phi(1/3)>0$ and $\Phi(1)>0$. This implies $\Phi(s)\geq 0$ throughout the entire interval $[1/3, 1]$. Consequently, we establish the bound $\mathcal{B}_f(z,\mathbf{r})\leq 1$ for $|a_0|\in [1/3, 1]$ and $\mathbf{r}=1/3n$.\vspace{1.2mm}

\noindent{\bf Case 2.} Let $a=|a_0|<1/3n$. Using (\ref{MA2}), (\ref{AM5}) and (\ref{AM12}) (setting $\mathbf{r}=1/3n$), we derive that
\begin{align*} &{\mathcal{B}_f(z,{\bf r})}\\&\leq \left(\frac{n{\bf r}+|a_0|}{1+|a_0|n{\bf r}}\right)^2+B\left(\frac{1}{3n}\right)+\frac{16}{9}\left({\bf r}^2\frac{(1-|a_0|^2)^2}{(1-|a_0|^2\;{\bf r}^2)^2}\right)+\lambda\left({\bf r}^2\frac{(1-|a_0|^2)^2}{(1-|a_0|^2\;{\bf r}^2)^2}\right)^2\nonumber\\&\leq
	\left(\frac{n{\bf r}+|a_0|}{1+|a_0|n{\bf r}}\right)^2+B\left(\frac{1}{3n}\right)+\frac{16}{9}\left({n^2\bf r}^2\frac{(1-a^2)^2}{(1-a^2\;{n^2\bf r}^2)^2}\right)+\lambda\left({n^2\bf r}^2\frac{(1-a^2)^2}{(1-a^2\;{n^2\bf r}^2)^2}\right)^2\nonumber\\&\leq
	\left(\frac{1+3a}{3+a}\right)^2+\frac{\sqrt{1-a^2}}{\sqrt{8}}+16\frac{(1-a^2)^2}{(9-a^2)^2}+81\lambda \frac{(1-a^2)^4}{(9-a^2)^4}\nonumber\\&= \Psi(a).\nonumber
\end{align*}
Routine and straightforward calculations show that the last expression for $\Psi(s)$ is an increasing function of $s$. Hence, its maximum is attained at the point $s_{0} = |a_{0}| = \frac{1}{3}$, and the maximum value $\Psi(s_{0})$ is found to be less than or equal to  $0.987\;(\leq 1)$. Therefore, the desired inequality (\ref{2.2}) follows for $|a_{0}| \in [0, 1/3)$. This proves that $\mathcal{B}_f(z, {\bf r}) \leq 1$ for $|a_{0}| \in [0, 1/3)$ and ${\bf r} \le 1/3n$.\vspace{1.2mm}

To prove that the constant $\lambda$ is sharp, we consider the holomorphic function $f_a$ in the polydisk $\mathbb{P}\Delta(0;1/n)$
given by (\ref{Eq-2.2}). Obviously
\begin{align*}
	f_a(z)=a-(1-a^2)\sum_{k=1}^{\infty} a^{k-1}(z_1+z_2+\cdots+z_n)^k
\end{align*}
for all $z\in\mathbb{P}\Delta(0;1/n)$. For the point $z=(r,r,\ldots,r)$,
we find that
\begin{align}\label{Eq-4.17}
|f_a(z)|=\dfrac{a+n{\bf r}}{1+na{\bf r}}\; \mbox{and}\;
f_a(z)=a-(1-a^2)\sum_{k=1}^{\infty} a^{k-1}(n{\bf r})^k.
\end{align}
Consequently, we see that
\begin{align*}
	{\mathcal{B}_{f_a}(z,r)}&=|f_a(z)|^2+\sum_{|\alpha|=1}^{\infty} |a_{\alpha}| |z|^{\alpha}+\frac{16}{9}\sum_{k=1}^{\infty}k\sum_{|\alpha|=k}|a_{\alpha}|^2\;r^{2\alpha}+\lambda \left(\sum_{k=1}^{\infty}k\sum_{|\alpha|=k}|a_{\alpha}|^2\;r^{2\alpha}\right)^2\nonumber\\&=
	\left(\frac{n{\bf r}+a}{1+an{\bf r}}\right)^2+\frac{(1-a^2)n{\bf r}}{1-na{\bf r}}+\frac{16}{9}\frac{(1-a^2)^2(n{\bf r})^2}{(1-a^2n^2{\bf r}^2)^2}+\lambda \left(\frac{(1-a^2)^2(n{\bf r})^2}{(1-a^2n^2{\bf r}^2)^2}\right)^2.
\end{align*}
Thus, it is clear that
\begin{align*}
	{\mathcal{B}_{f_a}\left(z,\frac{1}{3n}\right)}=\left(\frac{1+3a}{3+a}\right)^2 + \frac{1 - a^{2}}{3 - a}
	+16\frac{(1 - a^{2})^{2}}{(9 - a^{2})^{2}}
	+ 81\lambda\frac{ (1 - a^{2})^{4}}{(9 - a^{2})^{4}}.
\end{align*}
Choosing $a = s_0$, where $s_0$ is the positive root of $-513 + 910s + 80s^{2} + 2s^{3} + s^{4}=0$, results in $\mathcal{B}_{f_a}(z,1/3n)=1$. Additionally, we see that
\begin{align*}
	&{\mathcal{B}_{f_a}\left(z,\frac{1}{3n}\right)}\\&=\left(\frac{1+3a}{3+a}\right)^2 + \frac{1 - a^{2}}{3 - a}
	+16\frac{(1 - a^{2})^{2}}{(9 - a^{2})^{2}}
	+ 81\lambda_1\frac{ (1 - a^{2})^{4}}{(9 - a^{2})^{4}}\\&=
	\left(\frac{1+3a}{3+a}\right)^2 + \frac{1 - a^{2}}{3 - a}
	+16\frac{(1 - a^{2})^{2}}{(9 - a^{2})^{2}}
	+ 81\lambda\frac{ (1 - a^{2})^{4}}{(9 - a^{2})^{4}}
	+ 81(\lambda_{3} - \lambda)\frac{(1 - a^{2})^{4}}{(9 - a^{2})^{4}}.
\end{align*}
 Choose $a$ as the positive root $s_0$ of the equation $-513 + 910s + 80s^{2} + 2s^{3} + s^{4}=0$. Then, a simple computation shows that
\begin{align*}
{\mathcal{B}_{f_a}\left(z,\frac{1}{3n}\right)}=1+ (\lambda_{3} - \lambda)\frac{81(1 - a^{2})^{4}}{(9 - a^{2})^{4}},
\end{align*}
which is bigger than $1$ in the case when $\lambda_3>\lambda$. This verifies the sharpness of $\lambda$, and the proof is completed.
\end{proof}

\begin{proof}[\bf Proof of Theorem \ref{Th-2.3}] 
By the given condition $f(z)=a_0+\sum_{|\alpha|=1}^{\infty} a_{\alpha} z^{\alpha}$
is holomorphic in the polydisk $\mathbb{P}\Delta(0;1_n)$ such that $|f(z)|\leq 1$ in $\mathbb{P}\Delta(0;1_n)$. Let us take $z=(z_1,\ldots,z_n)\in \mathbb{P}\Delta(0;1_n)$ such that ${\bf r}=||z||_{\infty}$.\vspace{1.2mm}

To complete the proof, we consider following two cases:\vspace{1.2mm}

\noindent{\bf Case 1.} Let $a=|a_0|\geq {\bf r}$. From (\ref{AM5}), (\ref{AM7}) and (\ref{MA1}), it follows that
\begin{align*}
{\mathcal{C}_f(r)}&\leq \frac{n{\bf r}+|a_0|}{1+|a_0|n{\bf r}}+A({\bf r})+p\left({\bf r}^2\frac{(1-|a_0|^2)^2}{(1-|a_0|^2\;{\bf r}^2)^2}\right)\\&\leq
	\frac{n{\bf r}+|a_0|}{1+|a_0|n{\bf r}}+n{\bf r}\frac{1-|a_0|^2}{1-n|a_0|{\bf r}}+p\left({n^2\bf r}^2\frac{(1-a^2)^2}{(1-a^2\;{n^2\bf r}^2)^2}\right)=
	{\mathcal{C}_{f, 1}({\bf r})}.
\end{align*}
A routine computation shows that
\begin{align}\label{AAM1}{\mathcal{C}_{f, 1}\left(\frac{\sqrt{5}-2}{n}\right)}&=\frac{\sqrt{5}-2+a}{1+(\sqrt{5}-2)a}+(\sqrt{5}-2)\frac{1-a^2}{1-(\sqrt{5}-2)a}\\&\quad+p(9-4\sqrt{5})\frac{(1-a^2)^2}{(1-a^2(9-4\sqrt{5}))^2}\nonumber\\&=1+\frac{(1 - a)^{3}\!\left(7(-9 + 4\sqrt{5}) + 4(-47 + 21\sqrt{5})a + (-161 + 72\sqrt{5})a^{2}\right)}{
((4\sqrt{5} - 9)a^{2} + 1)^{2}}\nonumber\\&:=1+F(a).\nonumber
\end{align}
Since $-9 + 4\sqrt{5}<0$, $-47 + 21\sqrt{5}<0$ and $-161 + 72\sqrt{5}<0$, it follows that $F(a)\leq 0$. Hence, from (\ref{AAM1}), we see that 
\begin{align*}
	\mathcal{C}_{f, 1}\left(\frac{\sqrt{5}-2}{n}\right)\leq 1.
\end{align*} 
Thus, the desired inequality (\ref{2.3}) is valid for the case where $|a_0|\geq \mathbf{r}$ and $\mathbf{r}\leq {(\sqrt{5}-2)}/{n}$.\vspace{1.2mm}

\noindent{\bf Case 2.} Let $a=|a_0|<{\bf r}$. using (\ref{AM5}), (\ref{AM12}) and (\ref{MA1}), it is easy to see that
\begin{align*}
	{\mathcal{C}_f(r)}&\leq \frac{n{\bf r}+|a_0|}{1+|a_0|n{\bf r}}+B({\bf r})+p\left({\bf r}^2\frac{(1-|a_0|^2)^2}{(1-|a_0|^2\;{\bf r}^2)^2}\right)\nonumber\\&\leq
	\frac{n{\bf r}+|a_0|}{1+|a_0|n{\bf r}}+n{\bf r}\frac{\sqrt{1-|a_0|^2}}{\sqrt{1-n^2{\bf r}^2}}+p\left({n^2\bf r}^2\frac{(1-a^2)^2}{(1-a^2\;{n^2\bf r}^2)^2}\right)=
	{\mathcal{C}_{f, 2}({\bf r})}.\nonumber
\end{align*}
A routine calculation shows that ${\mathcal{C}_{f, 2}(\mathbf{r})}$ is increasing with respect to $\mathbf{r}$ for all $a\in[0,\mathbf{r}]$. Consequently, its maximum value is attained when $a=\mathbf{r}$. This specific case was already settled in the preceding analysis. \vspace{1.2mm}

To verify that the estimate for $p$ is sharp, let $a\in [0, 1)$ and consider the function $f_a$ as defined by \eqref{Eq-2.2}. Thus, in view of \eqref{Eq-4.17}, we obtain
\begin{align*}
{\mathcal{C}_{f_a}(r)}&=|f_a(z)|+\sum_{|\alpha|=1}^{\infty} |a_{\alpha}| |z|^{\alpha}+\left(2(\sqrt{5}-1)+\varepsilon\right)\left(\sum_{k=1}^{\infty}k\sum_{|\alpha|=k}|a_{\alpha}|^2\;r^{2\alpha}\right)\nonumber\\&=
\frac{n{\bf r}+a}{1+an{\bf r}}+\frac{(1-a^2)n{\bf r}}{1-na{\bf r}}+\left(2(\sqrt{5}-1)+\varepsilon\right)\left(\frac{(1-a^2)^2(n{\bf r})^2}{(1-a^2n^2{\bf r}^2)^2}\right).
\end{align*}
For ${\bf r}={(\sqrt{5}-2)}/{n}$, we have
\begin{align*}
&\mathcal{C}_{f_a}\left(\frac{\sqrt{5}-2}{n}\right)\\&=|f(z)|+\sum_{|\alpha|=1}^{\infty} |a_{\alpha}| |z|^{\alpha}+\left(2(\sqrt{5}-1)+\varepsilon\right)\left(\sum_{k=1}^{\infty}k\sum_{|\alpha|=k}|a_{\alpha}|^2\;r^{2\alpha}\right)\nonumber\\&=
\frac{a+\sqrt{5}-2}{1+(\sqrt{5}-2)a}+\frac{(\sqrt{5}-2)(1-a^2)}{1-(\sqrt{5}-2)a}+\left(2(\sqrt{5}-1)+\varepsilon\right)\left(\frac{(9-4\sqrt{5})(1-a^2)^2}{(1-(9-4\sqrt{5})a^2)^2}\right)\\&=
1+\frac{(1-a)^3 \left(a\left(\left(72\sqrt{5} - 161\right)+ \left(84\sqrt{5} - 188\right)\right)
+ 7\left(4\sqrt{5} - 9\right)\right)}{(1-(9-4\sqrt{5})a^2)^2}\\&\quad
+\varepsilon\frac{(9-4\sqrt{5})(1-a^2)^2}{(1-(9-4\sqrt{5})a^2)^2},
\end{align*}
which is bigger than $1$ in the case when $a\to 1$. This completes the proof.
\end{proof}






\end{document}